\documentclass[a4paper,11pt]{article}
\usepackage{latexsym}
\usepackage{amssymb}
\usepackage{amsfonts}
\usepackage{amsmath,amsthm}
\usepackage{mathrsfs}
\usepackage{bm}
\usepackage{enumerate}
\usepackage{stmaryrd}
\usepackage{xcolor}
\usepackage{commath}
\usepackage{array}
\usepackage{subcaption}
\usepackage{cases}
\usepackage{comment}
\usepackage{lipsum}
\usepackage{graphicx}
\usepackage{algorithmic}
\usepackage{algorithm}

\newcommand{\eps}{\varepsilon}

\newenvironment{@abssec}[1]{%
    \if@twocolumn

      \section*{#1}%
    \else

      \vspace{.05in}\footnotesize
      \parindent .2in
 {\upshape\bfseries #1. }\ignorespaces
    \fi}

    {\if@twocolumn\else\par\vspace{.1in}\fi}

\newenvironment{keywords}{\begin{@abssec}{\keywordsname}}{\end{@abssec}}

\newenvironment{AMS}{\begin{@abssec}{\AMSname}}{\end{@abssec}}

\newcommand\keywordsname{Key words}
\newcommand\AMSname{AMS subject classifications}
\newcommand\AMname{AMS subject classification}
\newcommand\restr[2]{{
\left.\kern-\nulldelimiterspace 
#1 
\vphantom{|} 
\right|_{#2} 
}}

\newtheorem{thm}{Theorem}[section]
\newtheorem{lem}[thm]{Lemma}

\newtheorem{rem}[thm]{Remark}
\newtheorem{dfn}[thm]{Definition}

\newtheorem{e.g.}{Example}

\def\XXint#1#2#3{{\setbox0=\hbox{$#1{#2#3}{\int}$}
\vcenter{\hbox{$#2#3$}}\kern-.5\wd0}}

\newcommand{\NN}{\mathbb{N}}
\newcommand{\ZZ}{\mathbb{Z}}

\newcommand{\RR}{\mathbb{R}}

\def\<{\langle }

\allowdisplaybreaks[1]

\title{\bf Visualizing Shape Functionals via Sinkhorn Multidimensional Scaling}

\author{Toshiaki Yachimura, Jun Okamoto, and Lorenzo Cavallina}
\date{}

\begin{document}

\maketitle

\begin{abstract}
In this paper, we present Sinkhorn multidimensional scaling (Sinkhorn MDS) as a method for visualizing shape functionals in shape spaces. This approach uses the Sinkhorn divergence to map these infinite-dimensional spaces into lower-dimensional Euclidean spaces. We establish error estimates for the embedding generated by Sinkhorn MDS compared to the unregularized case. Moreover, we validate the method through numerical experiments, including visualizations of the classical Dido’s problem and two newly introduced shape functionals: the double-well and Sinkhorn cone-type shape functionals. Our results demonstrate that Sinkhorn MDS effectively captures and visualizes shapes of shape functionals. 
\end{abstract}



\pagestyle{plain}
\thispagestyle{plain}

\section{Introduction}\label{sec:intro}
The so-called \textit{Dido’s problem}, inspired by the myth of the founding of Carthage as described by the Roman poet Virgil in the \textit{Aeneid}, is widely regarded as one of the earliest recorded problems in shape optimization. The problem asks the question: ``What shape $\Omega$ with a fixed perimeter $|\partial\Omega|$ maximizes the volume $|\Omega|$?". The solution to this problem is given by the disk in two dimensions (or the ball in higher dimensions). This simple problem has fascinated mathematicians for centuries, and it was not until the 20th century that rigorous proof was achieved in all dimensions (see \cite{degiorgi1958sulla, burago2013geometric} for some history and \cite{morgan2016geometric, bandle2017dido, cavallina2021double} and the references therein for some modern developments of the problem). Moreover, in connection to Dido’s problem, it is also known that the ball is not only the unique global maximizer but also the only critical point (this result is commonly referred to as Aleksandrov's \emph{soap bubble theorem} \cite{aleksandrov1958uniqueness, alexandrov1962characteristic}). Dido’s problem, as described above, can be expressed as an optimization problem for a shape functional (that is, a map that assigns a real number to each ``shape”, in this case, its perimeter) under a constraint. The study of optimal shapes under various constraints has long provided deep insights into geometric properties and the principles of the calculus of variations. Over time, this has expanded into the broader field of shape optimization, which involves determining the shapes that optimize certain functionals under given constraints (we refer the interested reader to the monographs \cite{delfour2011shapes, bendsoe2013topology, henrot2018shape, azegami2020shape} and the references therein). This includes not just problems of maximizing volume but also minimizing energy, optimizing material distribution, and other applications in engineering \cite{sokolowski1992introduction, allaire2019homogenization}. Foundational work on these topics has influenced a wide range of mathematical areas, including calculus of variations, geometric analysis, and partial differential equations \cite{kawohl2006rearrangements,  colding2011course, osserman2013survey, maggi2012sets, morgan2016geometric}.

In data science and machine learning, there is a growing interest in visualizing complex high-dimensional structures such as loss landscapes \cite{li2018visualizing, fort2019large}. A thorough understanding of the shape of these function landscapes is often considered crucial information for optimization problems. Such comprehensive knowledge can reveal the overall topography of loss functions, such as the location of their local minima, maxima, and saddle points. It is essential for guiding the design of efficient algorithms and understanding the convergence properties and stability of solutions. 

In contrast, the visualization of shape functionals is an even harder task. This is mainly due to the inherently complicated (nonlinear and nonconvex) structure of \textit{shape spaces} (see \cite[Chapter 1, Section 10]{delfour2011shapes}). Moreover, despite the lack of a universally accepted definition of \textit{shape space} (see \cite{kendall2009shape, arguillere2015general, welker2021suitable} and the references therein), the collection of all admissible shapes cannot be naturally endowed with the structure of a vector space (for instance, there is no straightforward way to define the sum of two shapes). Additionally, in many applications, it is useful to study the space of all possible local perturbations of a given shape and model it as an (open subset of an) infinite-dimensional vector space (in particular, this approach allows for endowing shape spaces with a distance \cite{Micheletti1972metrica, younes2010shapes} and a differential structure \cite[Chapter 9]{delfour2011shapes}). In other words, a shape space is endowed with a structure analogous to that of an infinite-dimensional manifold, which makes the visualization of shape functionals a very challenging task. 

Another common way of modeling shape spaces is to represent each shape by its characteristic function (see \cite[Chapter 5]{delfour2011shapes}) or by a uniform probability measure (see \cite[Chapter 4]{feydy2020geometric}, \cite[Section 9.6]{younes2010shapes}, \cite{schmitzer2013contour, liu2019least}). 
This approach allows the shape space to be viewed as a subspace of the space of probability measures and thus endow it with a natural distance. In particular, the space of shapes endowed with the Wasserstein distance is known to have a formal (infinite-dimensional) Riemannian structure  \cite{otto2001geometry}. 
In this paper, we adopt a similar approach by identifying each shape with the uniform probability measure on its boundary. In the subsequent numerical computations, where the shapes are the boundaries of bounded domains in $\mathbb{R}^2$, each uniform probability measure is further approximated by an empirical measure derived from a sampling of points. 

This paper presents a novel approach to visualizing shape functionals by projecting the infinite-dimensional shape space into a two-dimensional representation. The method leverages multidimensional scaling (MDS) \cite{borg2005modern, cox2000multidimensional, lim2021geometry,lim2022classical}. MDS is a classical technique used to embed data into a Euclidean space while preserving the pairwise distances or similarities among data points. This is particularly useful for visualizing complex relationships and structures that may not be apparent in the original high-dimensional space.

To apply MDS to shape spaces, we introduce an appropriate metric that captures the notion of distance or similarity between different shapes. This study utilizes optimal transport theory as a theoretical framework, specifically focusing on Wasserstein distances. Optimal transport theory provides a robust framework for computing these distances, which measure the cost of transporting one probability distribution to another. These distances are widely applicable in various fields, including economics, biology, computer graphics, and machine learning \cite{PeyreCutri}. 

However, calculating the Wasserstein distance is computationally expensive, especially in high-dimensional spaces. While methods like Wassmap \cite{Hamm2023} combine MDS with the Wasserstein distance for dimensionality reduction, they often face significant computational challenges. To address this, we propose the Sinkhorn MDS, based on the Sinkhorn divergence. The Sinkhorn divergence, derived from entropic optimal transport, removes the entropic bias introduced by the regularization and provides a more accurate estimate of the divergence between probability measures \cite{ramdas2017wasserstein, genevay2018learning, feydy2019interpolating}. It allows for stable and fast computations due to the smoothing effect of the entropic regularization. The Sinkhorn--Knopp algorithm \cite{sinkhorn1964relationship, SK1967} efficiently scales the transport matrix iteratively, and its implementation can take advantage of GPUs for further computational cost reductions.

In this paper, we first establish the error estimate for the embeddings obtained by Sinkhorn MDS compared to the case when the regularization parameter $\eps=0$, including Wassmap (see Theorem~\ref{Convergence rate of Sinkhorn MDS}). 

We then demonstrate the practical applications of Sinkhorn MDS in the visualization of shape functionals. A key application is the visualization of the volume functional under a fixed perimeter constraint (Dido's problem). In this context, a circle that maximizes volume serves as a reference point, and we consider perturbations of this optimal shape. Next, we introduce and visualize double-well shape functionals, which have two distinct minimizers. Finally, we examine Sinkhorn cone-type shape functionals and study the effects of increasing the regularization parameter $\eps > 0$. Through these examples, we validate the effectiveness of Sinkhorn MDS and demonstrate its ability to handle complex shape analysis and provide meaningful visualizations of shape functionals.

The paper is organized as follows: In Section \ref{sec:pre}, to properly introduce Sinkhorn MDS, we explain some preliminary technical aspects of MDS and optimal transport theory. In Section \ref{sec:SinkhornMDS}, we introduce Sinkhorn MDS and prove Theorem \ref{Convergence rate of Sinkhorn MDS} to establish the error estimate for the embeddings obtained by Sinkhorn MDS compared to the case when $\eps=0$. We provide numerical experiments in Section \ref{sec:Numerical_experiments} to demonstrate the effectiveness of our method.

\section{Preliminaries}\label{sec:pre}
This section explains the preliminary technical aspects of MDS and optimal transport theory.

\subsection{Classical multidimensional scaling}\label{subsec:cMDS}
In what follows, we first recall the definition and algorithm of classical multidimensional scaling (MDS) and introduce the results used in this paper. For further details, we refer the reader to the following references \cite{borg2005modern,lim2021geometry,lim2022classical,mardia2024multivariate}.

For $N \in \NN$, let $\mathcal{M}_N$ denote the set of all finite metric spaces with cardinality $N$. Specifically, we define
\begin{equation*}
\mathcal{M}_N := \left\{\mathcal{X} = (X,d_X) ~\middle|~ \# X = N \right\},
\end{equation*}
where $d_X$ is the distance function on $X$. The question of when a finite metric space $\mathcal{X}$ can be isometrically embedded into a Euclidean space $\RR^k$ and the constructive algorithm for finding such mappings are addressed by classical multidimensional scaling (MDS), which will be explained below.

For $\mathcal{X} = (X=\{x_1,...,x_N\}, d_X) \in \mathcal{M}_N$, we define the distance matrix $A_{\mathcal{X}}$ as
\begin{equation}\label{A}
(A_{\mathcal{X}})_{ij} := d_X^2(x_i,x_j) \quad (i,j = 1,\ldots,N).
\end{equation}
We also define the $N \times N$ matrix $B_{\mathcal{X}}$, 
known as the Gram matrix, as
\begin{equation}\label{B}
B_{\mathcal{X}} := -\frac{1}{2} H_N A_{\mathcal{X}} H_N,
\end{equation}
where $H_N$ is the centering matrix given by
\begin{equation*}
H_N := I_N - \frac{1}{N} \mathbf{1}_N \mathbf{1}_N^{T}. 
\end{equation*}
Here, $\mathbf{1}_N := (1,\ldots, 1)^T \in \RR^N$ is the column vector whose entries are all $1$. The matrix $H_N$ transforms the coordinate system by subtracting the mean from each variable, thus centering the data around the origin. A classical and important result for MDS is the following: 
\begin{thm}[Schoenberg's theorem \cite{MR1503248}]\label{Schoenberg’s Theorem}
Let $\mathcal{X} = (X = \{x_1, \ldots, x_N \}, d_{\mathcal{X}}) \in \mathcal{M}_N$, and suppose that $B_{\mathcal{X}}$ is the Gram matrix defined by \eqref{B}. Then, there exists some $k \in \mathbb{N}$ and a set of points $\{z_1, \ldots, z_N\} \subset \mathbb{R}^k$ in a $k$-dimensional Euclidean space such that $\mathcal{X}$ and $\{z_1, \ldots, z_N\} \subset \mathbb{R}^k$ are isometric if and only if the Gram matrix $B_{\mathcal{X}}$ is a positive semidefinite matrix.
\end{thm}
From Theorem \ref{Schoenberg’s Theorem}, one can see that the number of positive eigenvalues of the matrix $B_{\mathcal{X}}$ (and consequently, the number of negative eigenvalues) is crucial for isometrically embedding the metric space $\mathcal{X}$ into a Euclidean space. For the sake of the following discussion, let us define the number of positive eigenvalues as follows:
\begin{dfn}
For an $N \times N$ symmetric matrix $A$, let $\mathrm{pr}(A)$ denote the number of positive eigenvalues of $A$ (counting multiplicities).
\end{dfn}
Of course, not all finite metric spaces can be isometrically embedded into Euclidean space. However, according to Schoenberg's theorem, the obstacle to such embeddings is given by negative eigenvalues in the matrix $B_{\mathcal{X}}$. The algorithm of classical multidimensional scaling (MDS) considers embedding finite metric spaces into Euclidean spaces by eliminating these negative eigenvalues.
\begin{algorithm}[htbp]
\caption{Classical MDS Algorithm}
\begin{algorithmic}[1]
\STATE \textbf{Input:} A finite metric space $\mathcal{X} = (X, d_X) \in \mathcal{M}_N$.
\STATE Compute the distance matrix $A_{\mathcal{X}}$ and the double-centered distance matrix $B_{\mathcal{X}}$ as defined in \eqref{A} and \eqref{B}.
\STATE Compute the eigenvalues $\lambda_1 \geq \cdots \geq \lambda_{\mathrm{pr}(B_{\mathcal{X}})} > 0 = \lambda_{\mathrm{pr}(B_{\mathcal{X}})+1} \geq \cdots \geq \lambda_N$ of the matrix $B_{\mathcal{X}}$ and a corresponding orthonormal set of eigenvectors $\{v_1, \ldots, v_N\}$.
\STATE Choose $k \leq \mathrm{pr}(B_{\mathcal{X}})$.
\STATE Construct $\hat{\Lambda}_k^{1/2} V_k^T$ where $\hat{\Lambda}_k := \text{diag}(\lambda_1, \ldots, \lambda_k)$ and $V_k = [v_1 \ldots v_k]$.
\STATE Let $\Phi_{\mathcal{X},k}(x_i) \in \mathbb{R}^k$ be the $i$-th column vector of $\hat{\Lambda}_k^{1/2} V_k^T$ for each $i = 1, \ldots, n.$ 
\STATE \textbf{Output:} \begin{align*}
\Phi_{\mathcal{X},k}: X &\rightarrow \RR^k \\
x_i &\mapsto \hat{\Lambda}_k^{1/2} V_k^T(i).
\end{align*}
\end{algorithmic}
\end{algorithm}

\subsection{Fundamental tools for perturbation analysis in Sinkhorn MDS}
To establish the error estimate for the embeddings obtained by Sinkhorn MDS compared to the case when $\eps = 0$, we require two essential lemmas concerning the perturbation of eigenvalues and eigenvectors of symmetric matrices. These lemmas will be instrumental in our subsequent analysis. In what follows, the Euclidean norm and the inner product for vectors in $\mathbb{R}^k$ will be denoted by $\| \cdot \|_{\mathbb{R}^k}$ and $\langle \cdot, \cdot \rangle_{\mathbb{R}^k}$, respectively. For matrices, the operator norm will be denoted by $\| \cdot \|_2$, while the Frobenius norm by $\| \cdot \|_F$.
\begin{lem}[Eigenvalue perturbation theorem \cite{ikebe1987monotonicity}]\label{Perturbation theorem} 
Let $A$ and $B$ be $N \times N$ symmetric matrices, and let $C = A + B$. Let the eigenvalues of $A$, $B$, and $C$ be
\begin{equation*}
    \lambda_1 \geq \cdots \geq \lambda_N, \quad \mu_1 \geq \cdots \geq \mu_N, \quad \nu_1 \geq \cdots \geq \nu_N, 
\end{equation*}
respectively. Then,
\begin{equation*}
    |\nu_k - \lambda_k| \leq \|B\|_2 \leq \|B\|_F \quad 1 \leq k \leq N
\end{equation*}
holds.
\end{lem}

Furthermore, we introduce the following theorem for the estimate of eigenvector perturbations. 

\begin{lem}[Variant of the Davis–Kahan theorem \cite{MR3371006}]\label{V Davis--Kahan theorem}
Let $\Sigma$ and $\hat{\Sigma}$ be $N \times N$ symmetric matrices with eigenvalues $\lambda_1 \geq \cdots \geq \lambda_p$ and $\hat{\lambda}_1 \geq \cdots \geq \hat{\lambda}_p$, respectively. Fix integers $1 \leq r \leq s \leq p$ and assume that $\min \left\{\lambda_{r-1} - \lambda_r, \lambda_s - \lambda_{s+1} \right\} > 0$, where we define $\lambda_0 := \infty$ and $\lambda_{p+1} := -\infty$. Let $d := s - r + 1$, and let $V := (v_r, \ldots, v_s)$ and $\hat{V} := (\hat{v}_r, \ldots, \hat{v}_s)$ be matrices with orthonormal columns, where $v_j$ and $\hat{v}_j$ are the eigenvectors corresponding to the eigenvalues $\lambda_j$ and $\hat{\lambda}_j$, respectively, such that $\Sigma v_j = \lambda_j v_j$ and $\hat{\Sigma} \hat{v}_j = \hat{\lambda}_j \hat{v}_j$ for $j = r, r+1, \ldots, s$. Then, 
\begin{equation*}
\|\sin \Theta(\hat{V}, V)\|_{F} \leq \frac{2 \min \left\{d^{1/2} \|\hat{\Sigma} - \Sigma\|_2, \|\hat{\Sigma} - \Sigma\|_F \right\}}{\min \left\{\lambda_{r-1} - \lambda_r, \lambda_s - \lambda_{s+1} \right\}},
\end{equation*}
where $\Theta(\hat{V}, V) := \mathrm{diag} \left(\cos^{-1} \langle \hat{v}_r, v_r \rangle_{\RR^k}, \ldots, \cos^{-1} \langle \hat{v}_s, v_s\rangle_{\RR^k}  \right)$ and $\sin \Theta(\hat{V}, V) $ is defined entrywise.
\end{lem}

\subsection{Entropic optimal transport and Sinkhorn divergence}\label{subsec:EOTandSinkhorn}
For $N \in \mathbb{N}$, we define $\mathcal{P}_N$ as
\begin{equation}
    \mathcal{P}_N := \left\{\mu = (\mu_k) \in \mathbb{R}^N ~\middle|~ \mu_k \geq 0, ~ \sum_{k=1}^N \mu_k = 1 \right\}.
\end{equation}
For probability measures $\mu \in \mathcal{P}_N$ and $\nu \in \mathcal{P}_M$, we consider the following entropic optimal transport:
\begin{equation}\label{EOT}
    \mathrm{OT}_\eps(\mu,\nu) := \min_{\pi \in \Pi(\mu,\nu)} \langle C, \pi \rangle + \eps H(\pi),
\end{equation}
where $C \in \mathcal{M}_{N \times M}(\mathbb{R}_{+})$ is a cost matrix, and $\Pi(\mu,\nu)$ is the set of transportation plans from $\mu$ to $\nu$ defined as 
\begin{equation}
    \Pi(\mu,\nu) := \left\{\pi = (\pi_{ij}) \in \mathcal{P}_{N \times M} ~\middle|~  \sum_{l=1}^M \pi_{il} = \mu_i \text{ and } \sum_{l=1}^N \pi_{lj} = \nu_j \right\}.
\end{equation}
Here, $\eps > 0$ is the regularization parameter, and $H(\pi)$ denotes the Kullback--Leibler divergence \cite{KL51} of the transportation plan $\pi$ with respect to the product measure $\mu \otimes \nu$, which is defined as
\begin{equation}\label{KL}
    H(\pi) := \sum_{i=1}^N \sum_{j=1}^M \pi_{ij} \log \frac{\pi_{ij}}{\mu_i \nu_j},
\end{equation}
under the convention that $H(\pi):=+\infty$ if $\mu_i=0$ for some $i$ or $\nu_j=0$ for some $j$.

In the case where $\eps = 0$, problem \eqref{EOT} reduces to the following Monge--Kantorovich problem:
\begin{equation}\label{OT}
    \mathrm{OT}(\mu,\nu) := \min_{\pi \in \Pi(\mu,\nu)} \langle C, \pi \rangle.
\end{equation}
If the cost matrix $C$ is given by the squared Euclidean distances between the support points of $\mu$ and $\nu$, i.e., $C_{ij} = \|x_i - y_j\|^2$, then $\mathrm{OT}(\mu,\nu)$ corresponds to the squared Wasserstein distance between $\mu$ and $\nu$. 

One of the key advantages of entropic optimal transport is that it enables stable and fast computations using the Sinkhorn--Knopp algorithm \cite{sinkhorn1964relationship, SK1967}. This algorithm iteratively scales the transport matrix efficiently, and its implementation can be further optimized by leveraging GPUs, significantly reducing computational costs.

However, a known issue with entropic regularization is the entropy bias, where $\mathrm{OT}_{\eps}(\mu, \mu)$ is not zero. To address this issue, the Sinkhorn divergence was introduced \cite{ramdas2017wasserstein, genevay2018learning, feydy2019interpolating}, defined as:
\begin{equation}\label{SD}
SD_{\eps}(\mu,\nu) := \mathrm{OT}_\eps(\mu,\nu) - \frac{1}{2}\mathrm{OT}_\eps(\mu,\mu) - \frac{1}{2}\mathrm{OT}_\eps(\nu,\nu).
\end{equation}
The Sinkhorn divergence \eqref{SD} corrects the entropy bias by ensuring $SD_{\eps}(\mu,\mu) = 0$, providing a more accurate measure of distance between probability measures. As $\eps \to 0$, the Sinkhorn divergence \eqref{SD} converges to the optimal transport distance \eqref{OT}. In particular, the following quantitative error estimate holds: 
\begin{thm}[Error estimate of Sinkhorn divergence]\label{Estimate of SD}
For any $\eps > 0$ and probability measures $\mu \in \mathcal{P}_N$, $\nu \in \mathcal{P}_M$, the following estiamte holds:
\begin{equation*}
    |SD_{\eps}(\mu,\nu) - \mathrm{OT}(\mu,\nu)| \leq 2 \eps \log (NM).
\end{equation*}
\end{thm}

\begin{proof}
Let $\pi^{\eps} = (\pi^{\eps}_{ij}) \in \Pi(\mu,\nu)$ be the solution of the entropic optimal transport \eqref{EOT}, which is unique due to the strict convexity of $H$. We have
\begin{equation*}
\mathrm{OT}(\mu,\nu) = \min_{\pi \in \Pi(\mu,\nu)} \langle C, \pi \rangle 
\leq \langle C, \pi^{\eps} \rangle = \mathrm{OT}_\eps(\mu,\nu) - \eps H(\pi^{\eps}).  
\end{equation*}
This implies that
\begin{equation}\label{Entropy lower bound}
\mathrm{OT}_\eps(\mu,\nu) - \mathrm{OT}(\mu,\nu) \geq \eps H(\pi^{\eps}) \geq 0,
\end{equation}
where we used the fact that the Kullback--Leibler divergence \eqref{KL} is non-negative \cite[Theorem 2.3]{polyanskiy2024information}.
Next, let $\pi^* = (\pi^*_{ij}) \in \Pi(\mu,\nu)$ be the solution of the optimal transport \eqref{OT}. Then, we have
\begin{align*}
\mathrm{OT}_\eps(\mu,\nu) &= \min_{\pi \in \Pi(\mu,\nu)} \langle C, \pi \rangle + \eps H(\pi) \\
&\leq \langle C, \pi^* \rangle + \eps H(\pi^*) = \mathrm{OT}(\mu,\nu) + \eps H(\pi^*).  
\end{align*}
Thus, we obtain the following upper bound:
\begin{equation*}
\mathrm{OT}_\eps(\mu,\nu) - \mathrm{OT}(\mu,\nu) \leq \eps H(\pi^*). 
\end{equation*}
Here, considering the Kullback--Leibler divergence \eqref{KL}, we have
\begin{align*}
H(\pi^*) &= \sum_{ij} \pi^*_{ij} \log \frac{\pi^*_{ij}}{\mu_i \nu_j} \\
&= \sum_{i,j} \pi^*_{ij} \log \pi^*_{ij} - \sum_{i=1}^N \log \mu_i \sum_{j=1}^M \pi^*_{ij}  - \sum_{j=1}^M \log \nu_j \sum_{i=1}^N \pi^*_{ij}\\
&= \sum_{ij} \pi^*_{ij} \log \pi^*_{ij} - \sum_{i=1}^N  \mu_i \log \mu_i - \sum_{j=1}^M \nu_j \log \nu_j \\
&= \sum_{ij} \pi^*_{ij} \log \pi^*_{ij} + S(\mu) + S(\nu), 
\end{align*}
where $S(\mu)$ denotes the Shannon entropy, defined as $S(\mu) = - \sum_{i=1}^N \mu_i \log \mu_i$. By the concavity, Shannon entropy $S(\mu)$ is bounded by $\log N$ \cite[Theorem 1.4]{polyanskiy2024information}. Using this property, we can further estimate $H(\pi^*)$ as follows:
\begin{align*}
H(\pi^*) &= \sum_{i,j} \pi^*_{ij} \log \pi^*_{ij} + S(\mu) +  S(\nu) \\
&\leq S(\mu) +  S(\nu) \\
&\leq \log N + \log M = \log(NM). 
\end{align*}
Therefore, the upper bound becomes
\begin{equation}\label{Entropy upper bound}
\mathrm{OT}_\eps(\mu,\nu) - \mathrm{OT}(\mu,\nu) \leq \eps \log(NM). 
\end{equation}
Combining \ref{Entropy lower bound} and \ref{Entropy upper bound}, we obtain 
\begin{equation}\label{Entropy estimate}
|\mathrm{OT}_\eps(\mu,\nu) - \mathrm{OT}(\mu,\nu)| \leq \eps \log(NM). 
\end{equation}
Therefore, we have 
\begin{align*}
&|SD_{\eps}(\mu,\nu) - \mathrm{OT}(\mu,\nu) | \\
&\leq |\mathrm{OT}_\eps(\mu,\nu) - \mathrm{OT}(\mu,\nu)| + \frac{1}{2}|\mathrm{OT}_\eps(\mu,\mu) - \mathrm{OT}(\mu,\mu)| + \frac{1}{2}|\mathrm{OT}_\eps(\nu,\nu) - \mathrm{OT}(\nu,\nu)|\\
&\leq  \eps \log(NM) + \eps \log N + \eps \log M = 2 \eps \log (NM), 
\end{align*}
which completes the proof. 
\end{proof}

\section{Sinkhorn multidimensional scaling}\label{sec:SinkhornMDS}
In this section, we introduce Sinkhorn multidimensional scaling (Sinkhorn MDS), a method designed to visualize the “shape” of shape functionals within shape spaces using the Sinkhorn divergence \eqref{SD} between shapes as probability measures. By applying Sinkhorn MDS, we can effectively explore and capture the intricate structures of these shape functionals. We will define the Sinkhorn MDS algorithm and establish an error estimate for Sinkhorn MDS compared to the case when $\eps=0$.

\begin{dfn}[Sinkhorn MDS]\label{def Sinkhorn MDS}
For $\{n_i\}_{i=1}^N \subset \NN$ and a family of probability measures $\left\{ \mu_i  \mid \mu_i \in \mathcal{P}_{n_i}, ~i=1,\ldots,N \right\}$, define $A_{\eps} := \left(SD_\eps (\mu_i,\mu_j)\right)_{ij}$. Furthermore, let $B_{\eps}$ be the Gram matrix formed from $A_{\eps}$, defined as $B_{\eps} = -\frac{1}{2} H_N A_{\eps} H_N$. Let $\lambda_{1,\eps} \geq \cdots \geq \lambda_{N,\eps}$ be the eigenvalues of $B_{\eps}$, and $v_{1,\eps}, \ldots , v_{N,\eps}$ be a corresponding orthonormal set of eigenvectors. The Sinkhorn MDS mapping $\Phi_{\eps} : \{\mu_i\}_{i=1}^N  \to \RR^k$ is then defined as
\begin{equation*}
    \Phi_{\eps}(\mu_i) = \left(\lambda_{1,\eps}^{1/2} v_{1,\eps}(i), \cdots, \lambda_{k,\eps}^{1/2} v_{k,\eps}(i) \right)
\end{equation*}
\end{dfn}

The Sinkhron MDS algorithm can be summarized as follows: 
\begin{algorithm}[htbp]
\caption{Sinkhorn MDS}
\begin{algorithmic}[1]
\STATE \textbf{Input:} For $\{n_i\}_{i=1}^N \subset \NN$ and a family of probability measures $\{ \mu_i  \mid \mu_i \in \mathcal{P}_{n_i}, ~i=1,\ldots,N \}$. 
\STATE Compute the pairwise Sinkhorn divergence matrix $A_{\eps} := \left(SD_\eps (\mu_i,\mu_j)\right)_{ij}$ and the Gram matrix $B_{\eps} = -\frac{1}{2}H_N A_{\eps} H_N$. 
\STATE Compute the eigenvalues $\lambda_{1,\eps} \geq \cdots \geq \lambda_{\mathrm{pr}(B_{\eps}),\eps} > 0 = \lambda_{\mathrm{pr}(B_{\eps})+1,\eps} \geq \cdots \geq \lambda_{N,\eps}$ and a corresponding orthonormal set of eigenvectors $\{v_{1,\eps}, \ldots, v_{N,\eps}\}$ of the matrix $B_{\eps}$. 
\STATE Choose $k \leq \mathrm{pr}(B_{\eps})$.
\STATE Construct $\hat{\Lambda}_{k,\eps}^{1/2} V_{k,\eps}^T$ where $\hat{\Lambda}_{k,\eps} := \text{diag}(\lambda_{1,\eps}, \ldots, \lambda_{k,\eps})$ and $V_{k,\eps} = [v_{1,\eps} \ldots v_{k,\eps}]$.
\STATE Let $\Phi_{\eps}(x_i) \in \mathbb{R}^k$ be the $i$-th column vector of $\hat{\Lambda}_{k,\eps}^{1/2} V_{k,\eps}^T$ for each $i = 1, \ldots, n.$ 
\STATE \textbf{Output:} \begin{align*}
\Phi_{\eps}: \{\mu_i\}_{i=1}^N &\rightarrow \RR^k \\
\mu_i &\mapsto \hat{\Lambda}_{k,\eps}^{1/2} V_{k,\eps}^T(i).
\end{align*}
\end{algorithmic}
\end{algorithm}

\subsection{Error estimate of Sinkhorn MDS compared to the case when $\eps=0$}\label{subsec:main_result}
\begin{thm}[Error estimate of Sinkhorn MDS]\label{Convergence rate of Sinkhorn MDS}
For $\{n_i\}_{i=1}^N \subset \NN$ and a family of probability measures $\left\{ \mu_i  \mid \mu_i \in \mathcal{P}_{n_i}, ~i=1,\ldots,N \right\}$, let $B$ be the Gram matrix formed from $A= \left( \mathrm{OT}(\mu_i,\mu_j)\right)_{ij}$. Let $\lambda_1 \geq \cdots \geq \lambda_N$ be the eigenvalues of $B$ and let $v_{1}, \ldots, v_{N}$ be a corresponding orthonormal set of eigenvectors satisfying $\langle v_{j,\eps}, v_j \rangle_{\RR^k} \geq 0$ for $j=1, \ldots, N$, where the vectors $v_{j,\eps}$ are those that appear in Definition \ref{def Sinkhorn MDS}. For $k \leq \mathrm{pr}(B)$ such that $\lambda_k > \lambda_{k+1}$, we obtain the following error estimate: 
\begin{equation*}
\frac{1}{N}\sum_{i=1}^N\|\Phi_{\eps}(\mu_i) - \Phi(\mu_i)\|^2_{\RR^k} \leq \frac{64 \lambda_1 k N }{\left( \lambda_k - \lambda_{k+1} \right)^2} \left\{ \log \left(\max_i ~n_i \right)\right\}^2 \eps^2 + 4k \log \left(\max_i ~n_i \right) \eps.
\end{equation*}
\end{thm}

\begin{rem}
If the cost matrix $C$ is given by the squared Euclidean distances, then $A$ corresponds to the squared Wasserstein distance matrix, and $\Phi$ represents the Wassmap embedding \cite{Hamm2023} based on this squared Wasserstein distance matrix. 
\end{rem}

\begin{proof}
By constructing the embedding of Sinkhorn MDS $\Phi_{\eps}$, and the embedding $\Phi$ when $\eps = 0$, we have
\begin{align*}
\|\Phi_{\eps}(\mu_i) - \Phi(\mu_i)\|_{\RR^k}^2 &= \sum_{j=1}^k |\lambda^{1/2}_j v_j(i) - \lambda^{1/2}_{j,\eps} v_{j,\eps}(i)|^2 \\
&\leq \sum_{j=1}^k \left\{ \lambda^{1/2}_j |v_j(i) - v_{j,\eps}(i)| + |v_{j,\eps}(i)||\lambda^{1/2}_j  - \lambda^{1/2}_{j,\eps}| \right\}^2. 
\end{align*}
Applying the elementary inequality $(x+y)^2 \leq 2(x^2 + y^2)$, we have
\begin{equation*}
\|\Phi_{\eps}(\mu_i) - \Phi(\mu_i)\|_{\RR^k}^2 \leq 2 \sum_{j=1}^k \left\{\lambda_j |v_j(i) - v_{j,\eps}(i)|^2 + |v_{j,\eps}(i)|^2 |\lambda^{1/2}_j  - \lambda^{1/2}_{j,\eps}|^2 \right\}. 
\end{equation*}
Also, using the inequality $(x^{1/2} - y^{1/2})^2 \leq |x-y|$ for $x,y \geq 0$, we have 
\begin{equation*}
\|\Phi_{\eps}(\mu_i) - \Phi(\mu_i)\|_{\RR^k}^2 \leq 2\sum_{j=1}^k \left\{\lambda_j |v_j(i) - v_{j,\eps}(i)|^2 + |v_{j,\eps}(i)|^2 |\lambda_j  - \lambda_{j,\eps}| \right\}.
\end{equation*}
Applying Lemma \ref{Perturbation theorem}, we obtain 
\begin{align*}
\|\Phi_{\eps}(\mu_i) - \Phi(\mu_i)\|_{\RR^k}^2 &\leq 2\sum_{j=1}^k \left\{\lambda_j |v_j(i) - v_{j,\eps}(i)|^2 + |v_{j,\eps}(i)|^2 \|B_{\eps} - B\|_2 \right\} \\
&\leq  2 \lambda_1 \sum_{j=1}^k |v_j(i) - v_{j,\eps}(i)|^2 + 2\|B_{\eps} - B\|_2 \sum_{j=1}^k |v_{j,\eps}(i)|^2.
\end{align*}
Therefore, we have 
\begin{align*}
\sum_{i=1}^N \|\Phi_{\eps}(\mu_i) - \Phi(\mu_i)\|_{\RR^k}^2 &\leq  2 \lambda_1 \sum_{j=1}^k \|v_j - v_{j,\eps} \|_{\RR^k}^2 + 2k \|B_{\eps} - B\|_2.
\end{align*}
We used the normalization of the eigenvectors and, under the assumption that $\langle v_{j,\eps}, v_j \rangle_{\RR^k} \geq 0$ for any $j = 1, \ldots, k$, we obtain
\begin{align*} 
\sum_{j=1}^k \|v_{j,\eps} - v_j \|^2_{\RR^k} &= 2\sum_{j=1}^k \big( 1 - \langle v_{j,\eps}, v_j \rangle_{\RR^k} \big)\\
&\leq 2\sum_{j=1}^k \big( 1 - \langle v_{j,\eps}, v_j \rangle_{\RR^k} \big) \big(1 + \langle v_{j,\eps}, v_j \rangle_{\RR^k} \big)\\
&= 2\sum_{j=1}^k \sin^2 \big( \cos^{-1} \langle v_{j,\eps}, v_j \rangle_{\RR^k} \big).
\end{align*}
Let us denote $V_{\eps} = (v_{1,\eps}, \ldots, v_{k,\eps})$ and $V = (v_1, \ldots, v_k)$. Then, we have
\begin{equation*} 
\sum_{j=1}^k \sin^2 \big( \cos^{-1} \langle v_{j,\eps}, v_j \rangle_{\RR^k} \big) = \|\sin \Theta (V_{\eps}, V)\|_F^2. 
\end{equation*} 
Therefore, by applying Lemma \ref{V Davis--Kahan theorem} with $r=1$ and $s=k$, we obtain 
\begin{equation}\label{estimate A}
\begin{aligned} 
\sum_{j=1}^k \|v_{j,\eps} - v_j\|^2_{\RR^k} &\leq 2\|\sin \Theta (V_{\eps}, V)\|_F^2 \\ 
&\leq \frac{8 \min \left\{k^{1/2} \|B_{\eps} - B\|_2, \|B_{\eps} - B\|_F \right\}^2}{(\lambda_k - \lambda_{k+1})^2} \\
&\leq \frac{8 k}{(\lambda_k - \lambda_{k+1})^2} \|B_{\eps} - B\|_2^2. 
\end{aligned}
\end{equation}

We now need to estimate the norm $\|B_\eps - B\|_2$. By the definition of the centering matrix $H_N$, we have $\det(\lambda I - H_N) = \lambda (\lambda-1)^{N-1}$, and thus $\| H_N \|_2 = 1$. This implies that  
\begin{align*}
\|B_{\eps} -B\|_2 &= \frac{1}{2}\|H_N A_{\eps} H_N - H_N A H_N \|_2 \\
&=\frac{1}{2}\|H_N (A_{\eps} - A) H_N \|_2 \\
&\leq \frac{1}{2} \|A_{\eps} - A\|_2 \\
&\leq \frac{1}{2} \|A_{\eps} - A\|_F.
\end{align*}
Furthermore, applying Theorem \ref{Estimate of SD} to bound $\|A_\eps - A \|_F$, we get 
\begin{equation}\label{estimate B}    
\begin{aligned}
\|B_{\eps} -B\|_2 &\leq \frac{1}{2} \|A_{\eps} - A\|_F \\
&=\frac{1}{2} \sqrt{\sum_{i,j=1}^N \left(SD_{\eps}(\mu_i,\mu_j) - \mathrm{OT}(\mu_i,\mu_j) \right)^2} \\
&\leq \eps \sqrt{\sum_{i,j=1}^N \left(\log(n_i n_j)\right)^2} \\
&\leq 2 \eps N \log \left(\max_i ~n_i \right).
\end{aligned}
\end{equation}
Therefore, by combining \eqref{estimate A} and \eqref{estimate B}, we obtain
\begin{align*}
\sum_{i=1}^N \|\Phi_{\eps}(\mu_i) - \Phi(\mu_i)\|_{\RR^k}^2 &\leq  2 \lambda_1 \sum_{j=1}^k \|v_j - v_{j,\eps} \|^2 + 2k \|B_{\eps} - B\|_2 \\
&\leq \frac{16 \lambda_1 k}{\left( \lambda_k - \lambda_{k+1} \right)^2} \|B_{\eps} - B\|^2_2 + 2k \|B_{\eps} - B\|_2 \\
&\leq  \frac{64 \lambda_1 k N^2 }{\left(\lambda_k - \lambda_{k+1} \right)^2} \left\{ \log \left(\max_i ~n_i \right)\right\}^2 \eps^2 + 4k N \log \left(\max_i ~n_i \right) \eps, 
\end{align*}
which completes the proof. 
\end{proof}

\section{Numerical experiments}\label{sec:Numerical_experiments}
In this section, we describe the visualization method for shape functionals using Sinkhorn MDS. For $r \in C(\RR/ 2\pi\ZZ)$, we define $\Omega_r$ as the closed curve given given in polar coordinates by $r: \RR/ 2\pi \ZZ \to \RR$, as follows: 
\begin{equation*}
    \Omega_r := \{(r(\theta) \cos \theta, r(\theta) \sin \theta) \in \RR^2 \mid \theta \in\RR / 2\pi \ZZ\}.
\end{equation*}
We define the shape space $\mathcal{S}$ as the set of Jordan curves in $\RR^2$ that can be represented in polar coordinates. Specifically,
\begin{equation*}
    \mathcal{S} := \left\{\Omega_r \subset \RR^2 ~\middle|~ r \in C(\RR / 2\pi \ZZ), \Omega_r \,\, \text{is a Jordan curve} \right\}. 
\end{equation*}
Next, we explain the Fourier perturbation of shapes, which is essential for shape generation. First, choose some parameters $\delta > 0$, $a_0 \in \RR$, and $K \in \NN$. For $\{a_k\}_{k=1}^K \subset [-\delta, \delta]$, uniformly sampled from $[-\delta, \delta]$, we define a perturbation function as a random Fourier series: 
\begin{equation*}
\psi(\theta) := a_0 + \sum_{k=1}^K a_k \sin (k\theta).
\end{equation*}
Then, for $\Omega_r \in \mathcal{S}$,
\begin{equation}\label{random Fourier perturbed shape}
\Omega_{r+\psi} = \left\{ \left( \left\{r(\theta) + \psi(\theta)\right\} \cos \theta, \left\{r(\theta) + \psi(\theta)\right\} \sin (\theta) \right) \in \RR^2 \mid \theta \in\RR /  2\pi \ZZ\right\}    
\end{equation}
is defined as the random Fourier perturbed shape of $\Omega_r$. 

In numerical experiments, the shape is identified with a uniform probability measure. Let the probability measures corresponding to the shapes  $\Omega_1 \in \mathcal{S}$ and $\Omega_2 \in \mathcal{S}$ be $\mu_1$  and $\mu_2$, respectively. Then, the similarity $d_{\eps}$ between shapes $\Omega_1 \in \mathcal{S}$ and $\Omega_2 \in \mathcal{S}$ is defined as 
\begin{equation}\label{shape_similality}
d_{\eps}(\Omega_1, \Omega_2)^2 := SD_{\eps}(\mu_1, \mu_2). 
\end{equation}
Similarly, we define $d(\Omega_1, \Omega_2)^2 := \mathrm{OT}(\mu_1, \mu_2)$. In the context of these experiments, the cost matrix $C$ is assumed to be given by the squared Euclidean distances.

In the following, we present visualizations of three shape functionals. In the numerical experiment for each shape functional, we first generate basic shapes such as stationary points. (In each subsection, we will explain the definitions of these basic shapes.) Next, we apply random Fourier perturbations to these shapes to generate perturbed shapes \eqref{random Fourier perturbed shape}. Then, based on the similarity measure $d_{\eps}$ introduced in \eqref{shape_similality}, we apply Sinkhorn MDS to the generated shapes to project the shapes onto the two-dimensional $xy$-plane. Finally, we visualize the shape functionals by plotting the functional values of the shapes along the $z$-axis. First, we start with the visualization of the volume functional with perimeter constraint, corresponding to Dido's problem. Next, we introduce and visualize double-well type shape functionals that exhibit two minimizers. Lastly, we introduce Sinkhorn cone-type shape functionals and observe the effects of increasing $\eps > 0$. In the numerical experiments for visualizing these shape functionals, we generate shapes using Fourier parameters $K=5$. The method for choosing $\delta > 0$ varies depending on the scale of the shape functional, therefore, we will explain the appropriate parameters in each subsection.

\subsection{Volume maximization problems with perimeter constraints (Dido's problem)}\label{subsec:ex1}
First, we present numerical experiments focused on visualizing the volume functional under the perimeter constraint, a classical problem known as Dido’s problem, which was introduced in the introduction. For simplicity, we consider the functional $I : \mathcal{S} \to \RR$ defined as follows: 
\begin{equation*}
I(\Omega) := \frac{\mathrm{Vol}(\Omega)}{\mathrm{Per}(\Omega)^2}, 
\end{equation*}
where $\mathrm{Vol}(\Omega)$ denotes the volume of the region enclosed by $\Omega$, and $\mathrm{Per}(\Omega)$ denotes the perimeter of $\Omega$. The maximization problem of the functional $I$ under the perimeter constraint is equivalent to Dido's problem. 
\begin{figure}[htbp]
    \centering
    \includegraphics[width=0.5\linewidth]{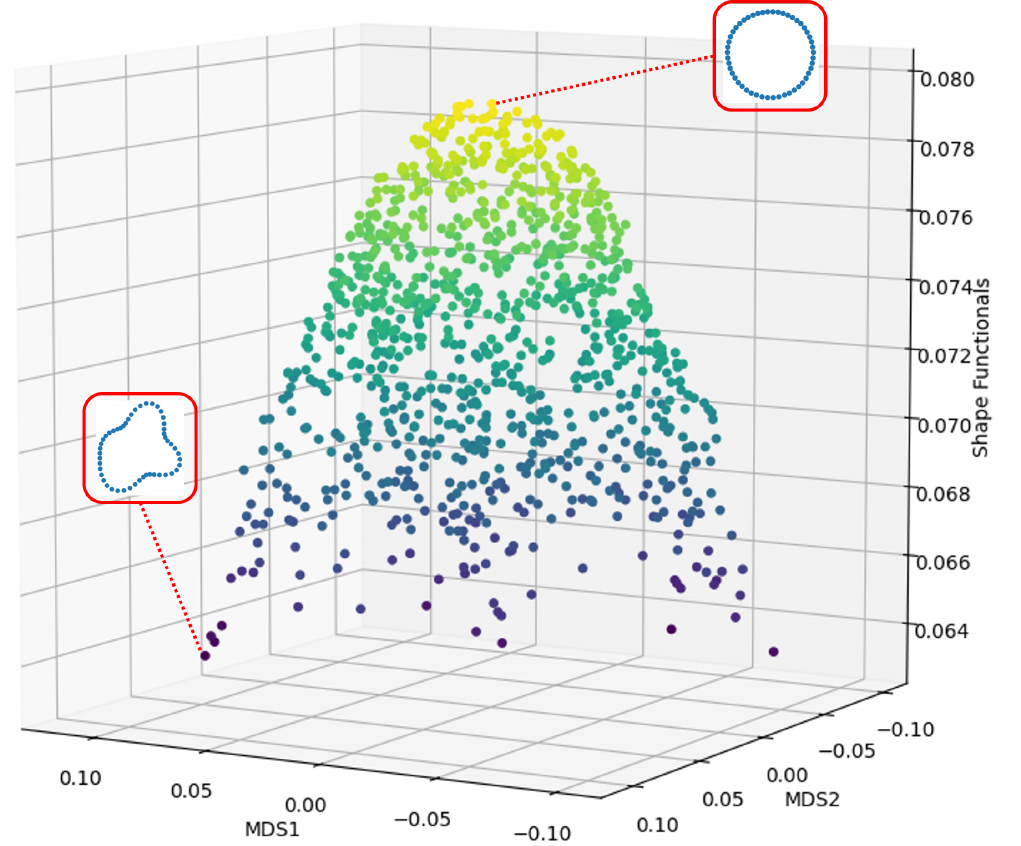}
    \caption{Visualization of the volume functional with perimeter constraint using Sinkhorn MDS with parameters: $1000$ shapes, $\eps = 10^{-3}$, and Fourier parameters $\delta=0.1$.}
    \label{fig1}
\end{figure}
In this numerical experiment, we set the basic shape as a circle, which maximizes the volume functional with perimeter constraint. We generate 1000 random Fourier perturbed shapes of the circle. 
Figure \ref{fig1} illustrates the results, revealing a non-trivial structure in the shape functional. This structure resembles a parabolic surface, with the circle as its unique critical point and maximizer. This visualization reflects Aleksandrov's \emph{soap bubble theorem} \cite{aleksandrov1958uniqueness, alexandrov1962characteristic}.

\subsection{Double-well type shape functionals}\label{subsec:ex2}
Next, we introduce the double-well type shape functionals and present numerical experiments on their visualization. The double-well type shape functional $V : \mathcal{S} \to \RR$ for shapes $A$ and $B$ is then defined by 
\begin{equation*}
V(\Omega) := d(\Omega, A)^2 d(\Omega, B)^2. 
\end{equation*}
This functional can be viewed as an extension to a shape functional of the double-well potential $V(x) = (x-a)^2 (x-b)^2$. In this numerical experiment, we visualize the double-well type shape functional with $A$ as a circle and $B$ as an equilateral triangle. 

\begin{figure}[htbp]
    \centering
    \includegraphics[width=0.5\linewidth]{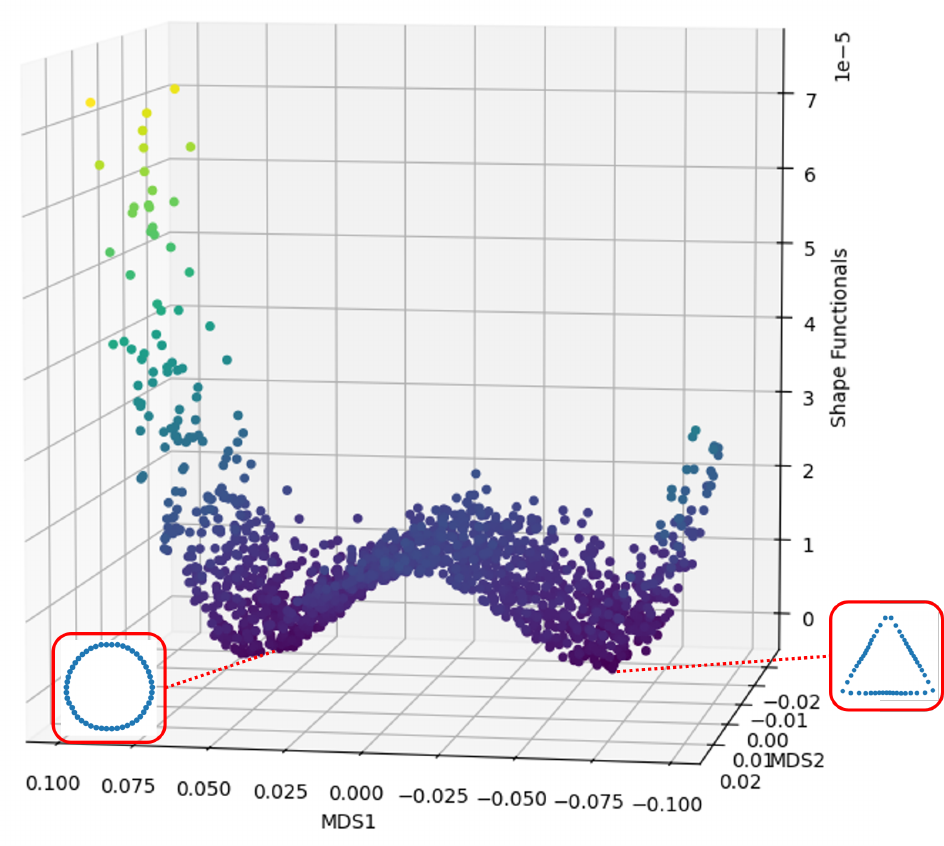}
    \caption{Visualization of the double-well type shape functionals using Sinkhorn MDS with parameters: $1600$ shapes, $\eps = 10^{-3}$, and Fourier parameters $\delta=0.03$.}
    \label{fig2}
\end{figure}

To generate the shapes, we identify the shapes $A$ and $B$ with their respective probability measures $\mu$ and $\nu$, and generate the McCann's displacement interpolation \cite{mccann2001polar} between $\mu$ and $\nu$. Specifically, let $T$ be the optimal transport map from $\mu$ to $\nu$. We define the family of shapes $\{\Omega_t\}_{t \in I}$ by $\Omega_t = ((1-t) \mathrm{Id} + t T)_\# \mu$, and consider them as the interpolated shapes between $A$ and $B$. Furthermore, we generate interpolated shapes with $I=[-0.3, 1.3]$ and a step size of $0.1$. For each interpolated shape, we generate $100$ random Fourier perturbed shapes, resulting in a total of $1600$ shapes. Figure \ref{fig2} illustrates the numerical experiment for the visualization of the double-well type shape functionals. We can observe the landscape around two critical points of a double-well type shape functionals. 

\subsection{Sinkhorn cone-type shape functionals and the effect of the regularization parameters}\label{subsec:ex3}
Finally, we introduce the Sinkhorn cone-type shape functional and examine the effect of the regularization parameter $\eps$ in the Sinkhorn MDS. The Sinkhorn cone-type shape functional $F_{\eps}: \mathcal{S} \to \RR$ with respect to a given shape $A$ is defined by
\begin{equation}\label{Sinkhorn cone-type shape functional}
    F_{\eps}(\Omega) := d_{\eps}(\Omega, A).
\end{equation}
Similar to subsection \ref{subsec:ex1}, we set the basic $A$ shape as a circle, which is the minimizer of the Sinkhorn cone-type shape functionals. We generate $1000$ random Fourier perturbed shapes of the circle. Figure \ref{fig3} illustrates the numerical experiment for the visualization of the Sinkhorn cone-type shape functionals. It is observed that as $\eps$ increases, the point cloud $ \left\{\big(\Phi_{\eps}(\Omega_i), F_{\eps}(\Omega_i)\big) \right\}_{i=1}^N$, which corresponds to the graph of the shape functional \eqref{Sinkhorn cone-type shape functional}, becomes more aligned. 
Since Sinkhorn divergence is known to interpolate between optimal transport and Maximum Mean Discrepancies (MMD), and converges to MMD as $\eps \to +\infty$ \cite{ramdas2017wasserstein, genevay2018learning, feydy2019interpolating}. these numerical results suggest that MMD has an effect in aligning the point cloud representing the shape functional.

\section{Conclusions}\label{sec:conclusions}
In this paper, we presented Sinkhorn multidimensional scaling (Sinkhorn MDS) as a tool for visualizing shape functionals within shape spaces. By combining classical multidimensional scaling (MDS) with the Sinkhorn divergence \eqref{SD}, we effectively mapped infinite-dimensional shape spaces into lower-dimensional Euclidean spaces, enabling the visualization of shape functionals. 

\begin{figure}[htbp]
    \begin{tabular}{cc}
      \begin{minipage}[t]{0.45\hsize}
        \centering
        \includegraphics[keepaspectratio, scale=0.25]{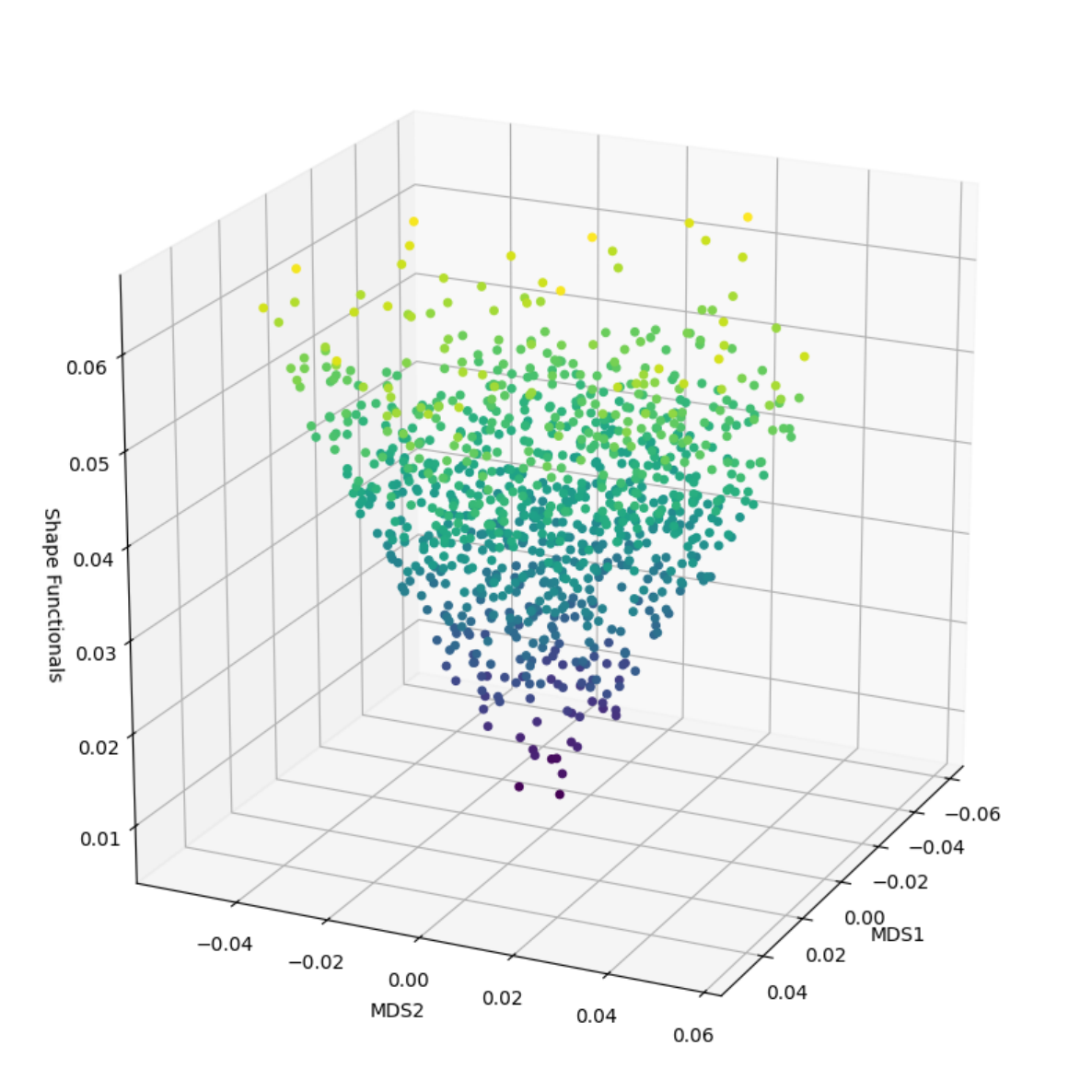}
        \subcaption{$\eps=0$}
        \label{shape_cone_eps=0}
      \end{minipage} &
      \begin{minipage}[t]{0.45\hsize}
        \centering
        \includegraphics[keepaspectratio, scale=0.25]{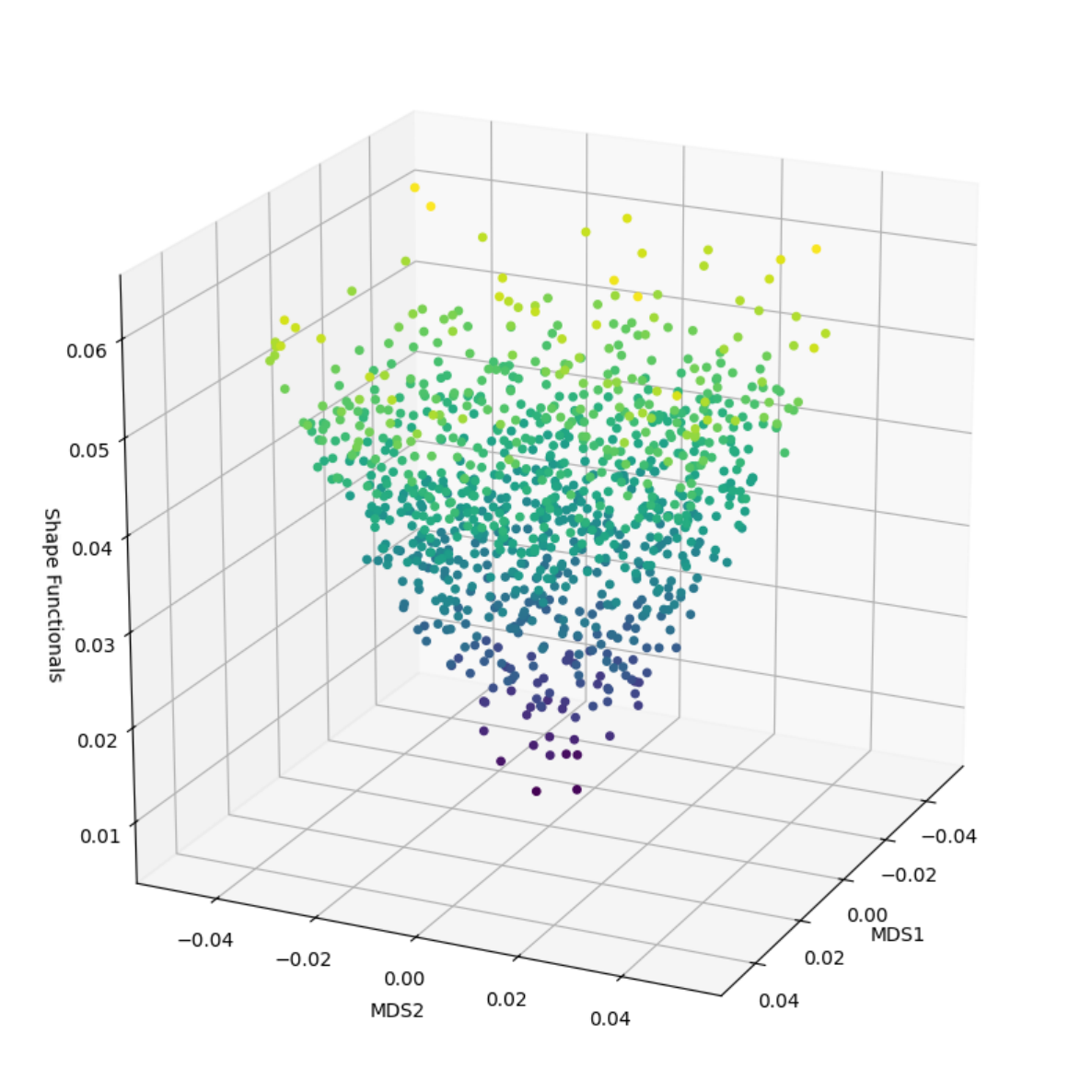}
        \subcaption{$\eps=0.01$}
        \label{shape_cone_eps=0.01}
      \end{minipage} \\
   
      \begin{minipage}[t]{0.45\hsize}
        \centering
        \includegraphics[keepaspectratio, scale=0.25]{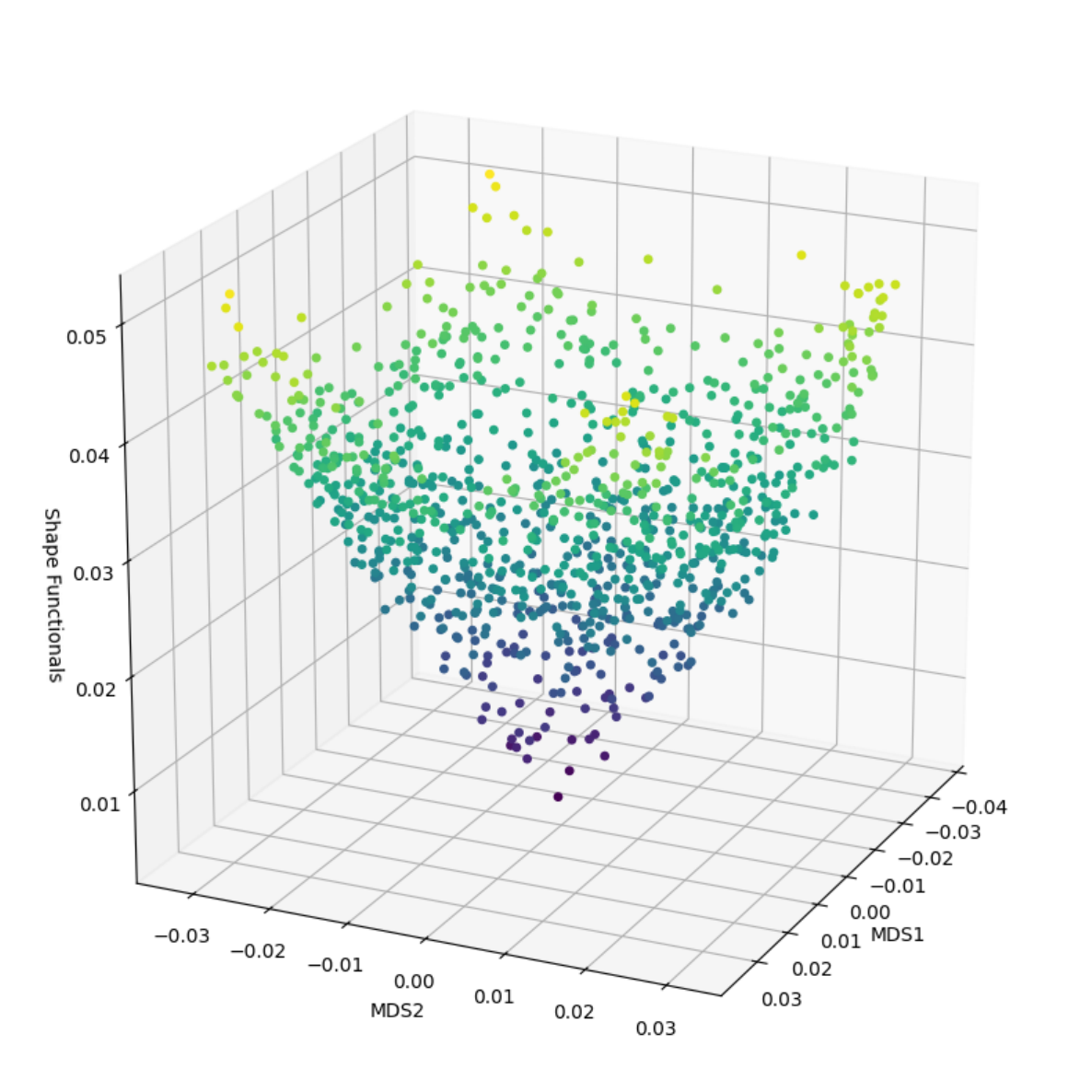}
        \subcaption{$\eps=0.1$}
        \label{shape_cone_eps=0.1}
      \end{minipage} &
      \begin{minipage}[t]{0.45\hsize}
        \centering
        \includegraphics[keepaspectratio, scale=0.25]{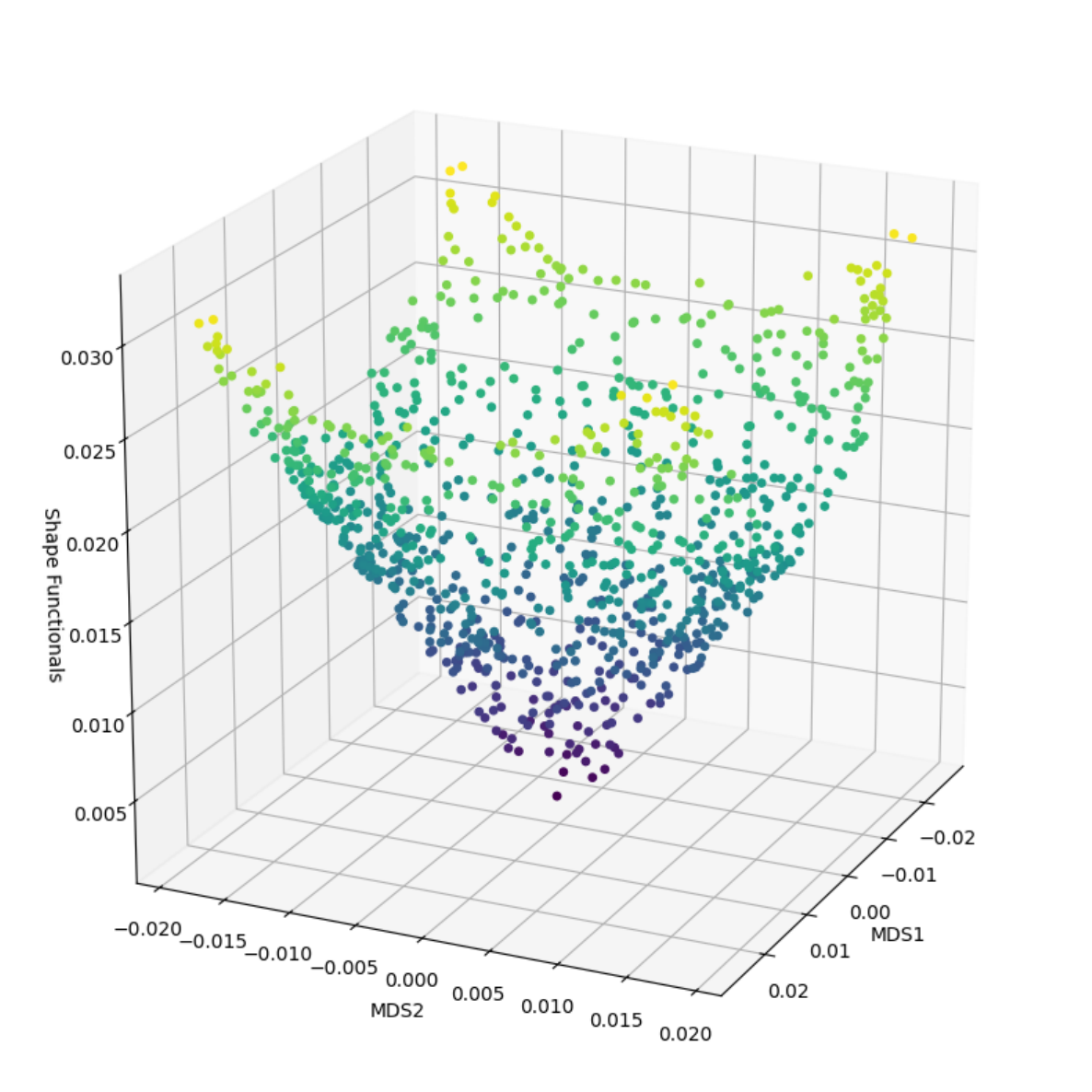}
        \subcaption{$\eps=1$}
        \label{shape_cone_eps=1}
      \end{minipage} 
    \end{tabular}
     \caption{Visualization of the Sinkhorn cone-type shape functionals using Sinkhorn MDS with parameters: $1000$ shapes, and Fourier parameters $\delta=0.06$.}
     \label{fig3}
\end{figure}

In particular, we established the error estimate for the embeddings obtained by Sinkhorn MDS compared to the case when $\eps=0$, including Wassmap \cite{Hamm2023}, as presented in Theorem \ref{Convergence rate of Sinkhorn MDS}. We also validated the effectiveness of Sinkhorn MDS through numerical experiments on shape optimization problems, including the classical Dido’s problem (Fig \ref{fig1}) and two newly introduced shape functionals: the double-well (Fig \ref{fig2}) and cone-type shape functionals (Fig \ref{fig3}). These experiments demonstrated that Sinkhorn MDS can accurately capture the complex structures of shape functionals and provide meaningful visualizations that reveal important geometric properties. 

Furthermore, we examined the effect of the regularization parameter $\varepsilon$ on the Sinkhorn cone-type shape functionals \eqref{Sinkhorn cone-type shape functional}. Notably, the numerical experiment (Fig \ref{fig3}) indicates that the shape functional becomes more aligned as the regularization parameter $\eps > 0$ increases. A mathematical proof of this phenomenon remains an open question for future research.

Overall, our results demonstrate that Sinkhorn MDS is an effective tool for shape analysis and visualization. This approach combines computational efficiency with theoretical rigor, making it suitable for various applications. Future research directions could include extending this method to more complex shape spaces, such as the space of knots, and exploring its application to practical problems in engineering such as topology optimization in elastic bodies.


\section*{Acknowledgments}
The first and second authors are partially supported by WPI-ASHBi at Kyoto University. This work was funded by JSPS Grant-in-Aid for Early-Career Scientists, 21K13822 (T.Y.), JST-Mirai Program, JPMJMI22G1 (J.O.), JSPS Grant-in-Aid for Transformative Research Areas (A), 23H04459 (L.C.) and JSPS Grant-in-Aid for Early-Career Scientists, 22K13935 (L.C.), JSPS Fund for the Promotion of Joint International Research (Fostering Joint International Research (B)) 21KK0044 (L.C.). 

\bibliographystyle{siam}
\bibliography{references}

\bigskip
\bigskip
\bigskip
\bigskip

\noindent
\textsc{Toshiaki Yachimura, \\ 
Mathematical Science Center for Co-creative Society, Tohoku University, 
Sendai 980-0845, Japan} \\
\noindent
{\em Electronic mail address:}
toshiaki.yachimura.a4@tohoku.ac.jp

\bigskip

\noindent
\textsc{Jun Okamoto, \\
Kyoto University Institute for Advanced Study, Sakyo-ku, Kyoto 606-8501, Japan}\\
\noindent
{\em Electronic mail address:}
okamoto.jun.8n@kyoto-u.ac.jp

\bigskip

\noindent
\textsc{Lorenzo Cavallina \\
Mathematical Institute, Tohoku University, Sendai 980-8578, Japan}\\
\noindent
{\em Electronic mail address:}
cavallina.lorenzo.e6@tohoku.ac.jp

\end{document}